\documentclass[a4paper, 11pt]{amsart}

\usepackage{epsfig}
\usepackage{amsthm,amsfonts}
\usepackage{amssymb,graphicx,color}
\usepackage[all]{xy}
\usepackage{cancel} 
\usepackage{verbatim}
\usepackage{hyperref}
\usepackage{enumitem}
\usepackage{comment}

\newtheorem{innercustomthm}{Theorem}
\newenvironment{customthm}[1]
  {\renewcommand\theinnercustomthm{#1}\innercustomthm}
  {\endinnercustomthm}
  
\newtheorem{innercustomcor}{Corollary}
\newenvironment{customcor}[1]
  {\renewcommand\theinnercustomcor{#1}\innercustomcor}
  {\endinnercustomcor}

\newtheorem{theorem}{Theorem}[section]
\newtheorem*{theorem*}{Theorem}
\newtheorem{lemma}[theorem]{Lemma}
\newtheorem{corollary}[theorem]{Corollary}
\newtheorem{proposition}[theorem]{Proposition}
\newtheorem{definition}[theorem]{Definition}
\newtheorem{example}[theorem]{Example}
\newtheorem{claim}{Claim}[theorem]

\newtheorem{conjecture}{Conjecture}

\newtheorem{problem}{Problem}

\newcommand{\C}{\mathbb{C}}

\newcommand{\Q}{\mathbb{Q}}

\begin{document}

\title[On the Arnold, Zariski and tang. Milnor fiber problems]
{On the Arnold corank, the Zariski multiplicity, and the tangent Milnor fiber problems}

\author[J. E. Sampaio]{Jos\'e Edson Sampaio}
\address{Jos\'e Edson Sampaio: Departamento de Matem\'atica, Universidade Federal do Cear\'a,
	      Rua Campus do Pici, s/n, Bloco 914, Pici, 60440-900, 
	      Fortaleza-CE, Brazil.} 
 \email{edsonsampaio@mat.ufc.br}



\thanks{The author was partially supported by CNPq-Brazil grants 303375/2025-6, 404611/2025-7 and 405618/2023-9. This work was supported by the Serrapilheira Institute (grant number Serra -- R-2110-39576).}

\begin{abstract}
In this article, we approach the Arnold corank problem posed by Arnold himself in 1975, which asks whether the corank of a holomorphic function is a topological invariant.  Here, the corank of a holomorphic function is the corank of its Hessian matrix at 0, which is one of the most basic and fundamental invariants of a holomorphic function. 
Here, we prove the topological invariance of the corank mod 2. Consequently, we obtain another proof of the Arnold corank problem in dimension 3, which was recently proved by Fernandes, Jelonek, and the author of this article.
We also prove that the corank is constant in a continuous family of holomorphic functions with isolated singularities and constant Milnor number.
We also approach a conjecture recently posed by Budur, Fernández de Bobadilla, Lê, and Nguyen (2022), which states that if two germs of holomorphic functions are right topologically equivalent, then the Milnor fibers of their initial forms are homotopy equivalent. 
We prove that a (much) weaker version of this conjecture, which states that if two holomorphic functions are right topologically equivalent, then the Milnor fibers of their initial forms have the same Euler characteristic, is equivalent to the Zariski multiplicity conjecture for right topological equivalence. Some other results are also presented here. For instance, the multiplicity $> 2$ of generalized semi-Bireskorn-Pham hypersurfaces is an invariant for the right topological equivalence. 
\end{abstract}

\maketitle

\tableofcontents
\section{Introduction}

Let $f\colon (\C^n,0)\to (\C,0)$ be the germ of a holomorphic function at the origin. Let $(V(f),0)$ be the germ of the zero set of $f$ at the origin. We recall that the {\bf corank} of the germ of a holomorphic function $f\colon (\C^n,0)\to (\C,0)$, with $\nabla f(0)=0$, is the corank of the Hessian matrix $f$ at $0$. We denote the corank of $f$ by ${\rm crk}(f)$. When $f$ is a reduced function, the {\bf corank} of $(V(f),0)$ is ${\rm crk}(V(f),0):={\rm crk}(f)$, and the {\bf rank} of $(V(f),0)$ is ${\rm rk}(V(f),0):=n-{\rm crk}(f)$. For simplicity, we write ${\rm rk}(f)$ instead of ${\rm rk}(V(f),0)$.

In 1975, Arnold asked whether the corank of a holomorphic function is an ambient topological invariant. More precisely, we have the following problem:
 
\begin{problem}[Arnold corank problem] 
Let $f,g\colon (\C^n,0)\to (\C,0)$ be germs of reduced holomorphic functions with a critical point at the origin. If there is a homeomorphism $\varphi\colon(\C^n,0)\to (\C^n,0)$ such that $\varphi(V(f))=V(g)$, is it true that ${\rm crk}(f)={\rm crk}(g)$?
\end{problem}
This problem is also stated in the book Arnold's Problems (see \cite[Problem 1975-14, p. 22]{Arnold:2004}). In this book, there is also a comment from Vassiliev that nothing is known to him about this problem (see \cite[p. 284]{Arnold:2004}).

Recently, Fernandes, Jelonek, and the author of this article proved in \cite{FernandesJS:2025} that the metric version of the Arnold corank problem has a positive answer, i.e., they proved that if there is a lipeomorphism $\varphi\colon(\C^n,0)\to (\C^n,0)$ such that $\varphi(V(f))=V(g)$, then ${\rm crk}(f)={\rm crk}(g)$. They also proved that the Arnold corank problem, when $n=3$, has a positive answer.

In this article, we prove that the Arnold corank mod 2 problem holds true.

\begin{customthm}{\ref*{thm:corank_mod_two}}
Let $f,g\colon (\C^n,0)\to (\C,0)$ be germs of reduced holomorphic functions with a critical point at the origin. Assume that there is a homeomorphism $\varphi\colon(\C^n,0)\to (\C^n,0)$ such that $\varphi(V(f))=V(g)$. Then ${\rm crk}(f)\equiv {\rm crk}(g) \mod 2$. 
\end{customthm}

We also present a proof, different from the one presented in \cite{FernandesJS:2025}, that the Arnold corank problem, when $n=3$, has a positive answer. To prove their result, the authors in \cite{FernandesJS:2025} used the embedded topological invariance of multiplicity 2 in 3 variables, announced by Aznar in \cite{Aznar:1980} (the authors in \cite{FernandesJS:2025} provided a proof of this fact), the Splitting Lemma for non-necessarily isolated singularities proved in \cite{GreuelP:2026}, and the famous classification of quotient singularities proved by Prill in \cite{Prill:1964}. Here, we still use the embedded topological invariance of multiplicity 2 in 3 variables, but we do not use the Splitting Lemma or the classification of quotient singularities.

We also obtain that the Arnold corank problem holds true in any continuous family of holomorphic functions $\{f_t\}_t$ such that the Milnor number $\mu(f_t)$ is independent of $t$ and is finite (see Theorem \ref{thm:corank_family}). Consequently, we obtain the following result in a purely algebraic setting:

\begin{customcor}{\ref*{rem:algebraic_corank}}
Let $k$ be a field of characteristic zero and $f\in k[t][\![ z_1,\dots,z_n ]\! ]$. Write $f_{t}=f(t,\cdot)\in k [\![ z_1,\dots,z_n ]\! ]$; so $t\mapsto f_{t}$ is an algebraic family of formal power series. Assume that the Milnor number of $f_{t}$ is finite, positive, and independent of $t$. Then the corank of $f_t$ is independent of $t$. 
\end{customcor}

Let $f\colon (\C^n,0)\to (\C,0)$ be the germ of a reduced holomorphic function at the origin. We recall that the \emph{multiplicity of} $V(f)$ at the origin, denoted by $m(V(f),0)$, is defined as follows: we write
$$f=f_m+f_{m+1}+\cdots+f_k+\cdots,$$
where each $f_k$ is the zero polynomial or a homogeneous polynomial of degree $k$, and $f_m\neq 0$. Then, $m(V(f),0):= m$. In this case, we also write $in(f):=f_m$.

In 1971 (see \cite{Zariski:1971}), Zariski asked whether the multiplicity of analytic hypersurfaces is an embedded topological invariant. This question, in its stated version, is the famous Zariski multiplicity conjecture.

\begin{conjecture}[Zariski multiplicity conjecture]
Let $f,g\colon (\C^n,0)\to (\C,0)$ be germs of reduced holomorphic functions. If there is a homeomorphism $\varphi\colon(\C^n,0)\to (\C^n,0)$ such that $\varphi(V(f))=V(g)$, then $m(V(f),0)=m(V(g),0)$.
\end{conjecture}

Recently, in \cite{BudurFLN:2020}, the following conjecture was proposed:

\begin{conjecture}\label{conj:Milnor_fiber_initial_form}
Let $f,g\colon(\mathbb{C}^n,0)\to (\mathbb{C},0)$ be two holomorphic function-germs. If there is a homeomorphism $\varphi\colon(\mathbb{C}^n,0)\to (\mathbb{C}^n,0)$ such that $f=g\circ \varphi$, then $in(f)^{-1}(1)$ and $in(g)^{-1}(1)$ are homotopy equivalent.
\end{conjecture}

More recently, the author of this article in \cite{Sampaio:2026} proposed the following conjectures, which are versions of the Zariski multiplicity Conjecture and Conjecture \ref{conj:Milnor_fiber_initial_form}.

\begin{conjecture}\label{conj:zariski_right}
Let $f,g\colon(\mathbb{C}^n,0)\to (\mathbb{C},0)$ be two reduced holomorphic function-germs. If there is a homeomorphism $\varphi\colon(\mathbb{C}^n,0)\to (\mathbb{C}^n,0)$ such that $f=g\circ \varphi$, then $m(V(f),0)=m(V(g),0)$.
\end{conjecture}

\begin{conjecture}\label{conj:Milnor_fiber_initial_form_homology}
Let $f,g\colon(\mathbb{C}^n,0)\to (\mathbb{C},0)$ be two holomorphic function-germs. If there is a homeomorphism $\varphi\colon(\mathbb{C}^n,0)\to (\mathbb{C}^n,0)$ such that $f=g\circ \varphi$, then $in(f)^{-1}(1)$ and $in(g)^{-1}(1)$ have the same homology.
\end{conjecture}

Here, we consider the following weaker conjecture.

\begin{conjecture}\label{conj:Milnor_fiber_initial_form_Euler_char}
Let $f,g\colon(\mathbb{C}^n,0)\to (\mathbb{C},0)$ be two holomorphic function-germs. If there is a homeomorphism $\varphi\colon(\mathbb{C}^n,0)\to (\mathbb{C}^n,0)$ such that $f=g\circ \varphi$, then $\chi(in(f)^{-1}(1))=\chi (in(g)^{-1}(1))$.
\end{conjecture}

In \cite{BobadillaP:2022}, it was proved that Conjecture \ref{conj:zariski_right} 
has a positive answer in the case of families of hypersurfaces with isolated singularities, and with that, it was proved that in the case of a topologically trivial family of hypersurfaces with isolated singularities, the Milnor fibers of the initial part have the same Borel-Moore cohomology (see \cite[Remark 8.2]{BobadillaP:2022}). In \cite{Sampaio:2026}, it was proved that if Conjecture \ref{conj:Milnor_fiber_initial_form_homology} has a positive answer, then Conjecture \ref{conj:zariski_right} also has a positive answer.

In the direction of proving that these two problems are equivalent, we prove the following:
\begin{customthm}{\ref*{thm:zariski_implies_tangent_conj_Euler_char}}
Let $f,g\colon(\mathbb{C}^n,0)\to (\mathbb{C},0)$ be two reduced holomorphic function-germs. Assume that there is a homeomorphism $\varphi\colon(\mathbb{C}^n,0)\to (\mathbb{C}^n,0)$ such that $f=g\circ \varphi$. If $m(V(f),0)=m(V(g),0)$, then $in(f)^{-1}(1)$ and $in(g)^{-1}(1)$ have the same Euler characteristic.
\end{customthm}
In particular, we obtain that if Conjecture \ref{conj:zariski_right} has a positive answer, then Conjecture \ref{conj:Milnor_fiber_initial_form_Euler_char} has a positive answer as well.
Indeed, we prove here that Conjectures \ref{conj:zariski_right} and \ref{conj:Milnor_fiber_initial_form_Euler_char} are equivalent.
\begin{customthm}{\ref*{thm:zariski_equiv_tangent_conj_Euler_char}}
Conjecture \ref{conj:zariski_right} has a positive answer if and only if Conjecture \ref{conj:Milnor_fiber_initial_form_Euler_char} has a positive answer.
\end{customthm}

Coming back to the Zariski multiplicity problem, it was proved in \cite{YoshinagaS:1978} that the multiplicity of Brieskorn-Pham hypersurfaces is invariant under ambient homeomorphisms, i.e., if $X,Y\subset \C^n$ are two Brieskorn-Pham hypersurfaces such that there is a homeomorphism $h\colon (\C^n,X,0)\to (\C^n,Y,0)$, then $m(X,0)=m(Y,0)$. 

Here, we deal with the following more general class of singular hypersurfaces introduced in \cite{FernandesJS:2025}:

\begin{definition}
A holomorphic function $f\colon (\C^n,0)\to (\C,0)$
is said a {\bf generalized semi-Brieskorn-Pham} if there are a holomorphic diffeomorphism $\phi\colon (\C^n,0)\to (\C^n,0)$, a holomorphic function $h\colon (\C^n,0)\to (\C,0)$, and positive integer numbers $k, a_1,...,a_{k-1}$ and $a_k$ such that 
\begin{enumerate}
\item $a_1\leq a_2\leq ....\leq a_k$;
\item ${\rm ord}_0h>a_k$;
 \item $f\circ \phi(z_1,...,z_n)=z_1^{a_1}+...+z_k^{a_k}+h(z_1,...,z_n)$ for all $z=(z_1,...,z_n)$ close to $0$.
\end{enumerate}
In this case, if $a_1=1$, we define $m_1(f)=1$, and if $a_1>1$, we define $m_1(f):=\max\{i;a_i=a_1\}$. When $f$ is a generalized semi-Brieskorn-Pham function, we also say that $V(f)$ is a {\bf generalized semi-Brieskorn-Pham hypersurface}.
\end{definition}

Any Brieskorn-Pham polynomial is a generalized semi-Brieskorn-Pham function. However, the class of generalized semi-Brieskorn-Pham hypersurfaces is much larger than the class of Brieskorn-Pham hypersurfaces. For instance, a generalized semi-Brieskorn-Pham hypersurface does not necessarily have isolated singularities. 
\begin{example}\label{exam:linear_cone}
\begin{itemize}
    \item [a)] Any holomorphic function $f\colon (\C^n,0)\to (\C,0)$ such that $C(V(f),0)$ is a linear subspace is a generalized semi-Brieskorn-Pham function;
    \item [b)] Any irreducible holomorphic function $f\colon (\C^2,0)\to (\C,0)$ is a generalized semi-Brieskorn-Pham function;
    \item [c)] If $f\colon (\C^n,0)\to (\C,0)$ is a generalized semi-Brieskorn-Pham function, then $\tilde f\colon (\C^{n+1},0)\to (\C,0)$ given by $\tilde f(t,x)=t^m+f(x)$ is a generalized semi-Brieskorn-Pham function as well. In particular, the surfaces of the type $\{(x,y,z)\in \C^3;x^m+f(y,z)=0\}$ with $f\colon (\C^2,0)\to (\C,0)$ being an irreducible holomorphic function, studied by Mendris and N\'emethi in \cite{MendrisN:2005}, are generalized semi-Brieskorn-Pham hypersurfaces.
\end{itemize}
\end{example}

If $f,g\colon (\C^n,0)\to (\C,0)$ are generalized semi-Brieskorn-Pham functions such that $m(V(f),0),m(V(g),0)\not=2$ and there is a homeomorphism $\phi\colon (\C^n,0)\to (\C^n,0)$ with $\phi(V(f))=V(g)$, then we obtain that $m(V(f),0)=m(V(g),0)$ and $m_1(f)=m_1(g)$ (see Proposition \ref{prop:semi_top_mult}).

As a consequence, we obtain that the multiplicity of generalized semi-Brieskorn-Pham hypersurfaces in $\C^3$ is a topological invariant (see Corollary \ref{cor:inv_mult_gsBP_surfaces}).

In particular, we obtain the result proved by Mendris and N\'emethi \cite[Corollary 2]{MendrisN:2005}, which states that the multiplicity of surfaces of the type $\{(x,y,z)\in \C^3;x^m+f(y,z)=0\}$, where $f\colon (\C^2,0)\to (\C,0)$ is an irreducible holomorphic function, is a topological invariant (see Corollary \ref{cor:inv_mult_Mendris-Nemethi_surfaces}).

Note that if $f,g\colon(\mathbb{C}^n,0)\to (\mathbb{C},0)$ are  holomorphic function-germs with isolated singularities and there is a homeomorphism $\psi\colon(\mathbb{C}^n,0)\to (\mathbb{C}^n,0)$ such that $\psi(V(f))=V(g)$, then using results of King and Saeki (see Corollary in \cite{King:1978} and Corollary 2 in \cite{Saeki:1989}), there is a homeomorphism $\varphi\colon(\mathbb{C}^n,0)\to (\mathbb{C}^n,0)$ such that $f=g\circ \varphi$ or $f=g_1\circ \varphi$, where $g_1=c_1\circ g\circ c_n$, and $c_k\colon \C^k\to \C^k$ is the conjugation mapping given by $c_k(z_1,...,z_k)=(\overline{z_1},...,\overline{z_k})$.
Note also that ${\rm crk} (g)={\rm crk} (g_1)$, $m(V(g),0)=m(V(g_1),0)$, and moreover, $in(g)^{-1}(1)$ and $in(g_1)^{-1}(1)$ are  diffeomorphic. Thus, when we are only dealing with holomorphic functions with isolated singularities, the results presented here for topological right equivalence also hold for topological $V$-equivalence.

\bigskip

\noindent{\bf Acknowledgements}. The author would like to thank Alexandre Fernandes, Euripedes da Silva, and Joserlan da Silva for their interest in this research.

\section{Preliminaries}

\noindent {\bf Notation:} 
Let $f\colon  (\mathbb{C}^{n}, 0) \to (\mathbb{C}, 0)$ be a holomorphic germ. We denote $V(f)=f^{-1}(0)$ and $F_f$ as the Milnor fiber of $f$ at $0$. When $f$ is a homogeneous polynomial, $F_f=f^{-1}(1)$. Let $m=\text{ord}_0f$ denote the order of $f$ at the origin. Let $in(f)$ denote the initial homogeneous polynomial of degree $m$ in the Taylor expansion of $f$. The multiplicity of $V(f)$ at the origin is denoted by $m(V(f),0)$, which coincides with $\text{ord}_0f$ whenever $f$ is the germ of a reduced function.

The monodromy of the Milnor fibration of $f$ (at $0$) is $h_f\colon F_f\to F_f$. Moreover, $\Lambda(h_f^k)$ denotes the Lefschetz number of $h_f^k$ (the composition of the monodromy of $f$ with itself $k$ times). 

We denote the projectivized tangent cone of $V(f)$ by $PV(in(f)) \subset \mathbb{C}P^n$, defined as the projectivized zero locus of the initial part:
\[ PV(in(f)) := \{ [z] \in \mathbb{C}P^{n-1} \mid in(f)(z) = 0 \}. \]

\subsection{Some results from A'Campo and L\^e}

The following result was proved separately by A'Campo in \cite{ACampo:1973} and L\^e in \cite{Le:1973}.
\begin{theorem}[A'Campo-L\^e's Theorem]
Let $f,g\colon(\mathbb{C}^n,0)\to (\mathbb{C},0)$ be two complex analytic functions such that there is a homeomorphism $\varphi\colon(\mathbb{C}^n,V(f),0)\to (\mathbb{C}^n,V(g),0)$. If $m(V(f),0)=1$ then $m(V(g),0)=1$ as well.
\end{theorem}

The next result was proved by A'Campo in \cite{ACampo:1975} (see also \cite{Karras:1986}). \begin{theorem}\label{thm:acampo_complement}
Let $f\colon(\mathbb{C}^n,0)\to (\mathbb{C},0)$ be a complex analytic function. Then
$$
m(V(f),0)\leq \inf \{k>0;\Lambda(h_f^k)\not=0\}.
$$
Moreover, if $\chi(\C P^{n-1}\setminus PV(in(f)))\not=0$, then
$$
m(V(f),0)=\inf \{k>0;\Lambda(h_f^k)\not=0\}.
$$
\end{theorem}

The following result follows from \cite[Theorem 3.3 and Remark 3.4]{Le:1973a}:

\begin{theorem}\label{thm:monodromy}
Let $f,g : \mathbb C^{n}\to  \mathbb C$ be two reduced holomorphic  functions. If the germs $(V(f), 0)$ and $(V(g),0)$ are ambient  homeomorphic, then the induced maps in homology of the monodromies of $f$ and $g$ at $0$ are conjugate. In particular, for each positive integer $k$, $\Lambda(h_f^k)=\Lambda(h_g^k)$, and $\chi(F_f ) = \chi(F_g )$.
\end{theorem}

\section{Tangent Milnor fiber conjecture and Zariski multiplicity conjecture}
\subsection{Conjecture \ref{conj:zariski_right} implies Conjecture \ref{conj:Milnor_fiber_initial_form_Euler_char}}
\begin{theorem}\label{thm:zariski_implies_tangent_conj_Euler_char}
Let $f,g\colon(\mathbb{C}^n,0)\to (\mathbb{C},0)$ be two reduced holomorphic function-germs. Assume that there is a homeomorphism $\varphi\colon(\mathbb{C}^n,0)\to (\mathbb{C}^n,0)$ such that $f=g\circ \varphi$. If $m(V(f),0)=m(V(g),0)$, then $in(f)^{-1}(1)$ and $in(g)^{-1}(1)$ have the same Euler characteristic.
\end{theorem}
Indeed, Theorem \ref{thm:zariski_implies_tangent_conj_Euler_char} is a direct consequence of the following result that relates the Lefschetz number of the $m$-th iterate of the monodromy to the Euler characteristic of the Milnor fiber of the initial term $in(f)$.
\begin{proposition}\label{prop:formula_euler_charac_monodromy}
Let $f: (\mathbb{C}^{n}, 0) \to (\mathbb{C}, 0)$ be a reduced holomorphic function-germ and $m = m(V(f), 0)$. Let $h_f$ be the monodromy of the Milnor fiber of $f$, and let $F_{in(f)}$ be the Milnor fiber of the initial homogeneous part of $f$. Then
\[ \Lambda(h_f^m) = \chi(F_{in(f)}). \]
\end{proposition}
\begin{proof}

Let $\sigma\colon  X_1 \to \mathbb{C}^{n}$ be the blow-up of the origin, and let $E_0 \cong \mathbb{C}P^{n-1}$ be the resulting exceptional divisor.
Let $\pi: X \to \mathbb{C}^{n}$ be a resolution of singularities extending the blow-up of the origin and where the total transform divisor $D = \pi^{-1}(V(f))$ has normal crossings. The stratum $S_k$ is the set of points $x \in \pi^{-1}(0)$ where the local equation of $D$ is a pure power of the form $z^k = 0$. 
This definition excludes points where the divisor has normal crossings (e.g., $z_1^a z_2^b = 0$). Thus, $S_k$ is locally closed and contains no intersection points of components of the divisor. It is easy to obtain the following:

\begin{claim}\label{claim:complement_proj_tg_cone}
For $m = m(V(f), 0)$, the stratum $S_m$ is diffeomorphic to the complement of the projectivized tangent cone:
\[ S_m \cong \mathbb{C}P^{n-1} \setminus PV(in(f)). \]
Furthermore, $S_k = \emptyset$ for all $k < m$.
\end{claim}

We use the trace formula for the Lefschetz number established in \cite[Theorem 1]{ACampo:1975}:
\[ \Lambda(h_f^k) = \sum_{d | k} d \cdot \chi(S_d). \]
We evaluate this for $k=m$. By Claim \ref{claim:complement_proj_tg_cone}, $S_d = \emptyset$ for all $d < m$. Therefore, the sum reduces to the single term $d=m$:
\begin{equation} \label{eq:trace}
    \Lambda(h_f^m) = m \cdot \chi(S_m).
\end{equation}
From Claim \ref{claim:complement_proj_tg_cone}, we have
\begin{equation} \label{eq:geom_subs}
    \Lambda(h_f^m) = m \cdot \chi(\mathbb{C}P^{n-1} \setminus PV(in(f))).
\end{equation}

Now, consider the Milnor fiber of the homogeneous polynomial $in(f)$:
\[ F_{in(f)} = \{ z \in \mathbb{C}^{n} \mid in(f)(z) = 1 \}. \]
Define the projection $\rho\colon F_{in(f)} \to \mathbb{C}P^{n-1} \setminus PV(in(f))$ by $\rho(z) = [z]$.
For any $[z_0] \in \mathbb{C}P^{n-1} \setminus PV(in(f))$, we have $in(f)(z_0) \neq 0$. The points in the fiber $\rho^{-1}([z_0])$ are vectors $\lambda z_0$ satisfying:
\[ in(f)(\lambda z_0) = \lambda^m \cdot in(f)(z_0) = 1 \implies \lambda^m = \frac{1}{in(f)(z_0)}. \]
This equation has exactly $m$ distinct roots for $\lambda$. Thus, $\rho$ is an unramified covering map of degree $m$.
The Euler characteristic of the total space of a finite covering is the degree times the Euler characteristic of the base:
\[ \chi(F_{in(f)}) = m \cdot \chi(\mathbb{C}P^{n-1} \setminus PV(in(f))). \]
Comparing this with equation (\ref{eq:geom_subs}), we obtain the desired result:
\[ \Lambda(h_f^m) = \chi(F_{in(f)}). \]
\end{proof}

\begin{proof}[Proof of Theorem \ref{thm:zariski_implies_tangent_conj_Euler_char}]
Using that $m=m(V(f),0)=m(V(g),0)$, by Theorem \ref{thm:monodromy}, $ \Lambda(h_f^m)= \Lambda(h_g^m)$. By Proposition \ref{prop:formula_euler_charac_monodromy}, $\Lambda(h_f^m) = \chi(F_{in(f)})$ and $\Lambda(h_g^m) = \chi(F_{in(g)})$. Therefore, $\chi(F_{in(f)})=\chi(F_{in(g)})$.
\end{proof}

\subsection{Conjectures \ref{conj:zariski_right} and \ref{conj:Milnor_fiber_initial_form_Euler_char} are equivalent}

\begin{theorem}\label{thm:zariski_equiv_tangent_conj_Euler_char}
Conjecture \ref{conj:zariski_right} has a positive answer if and only if Conjecture \ref{conj:Milnor_fiber_initial_form_Euler_char} has a positive answer.
\end{theorem}
\begin{proof}
By Theorem \ref{thm:zariski_implies_tangent_conj_Euler_char}, we only need to prove that if Conjecture \ref{conj:Milnor_fiber_initial_form_Euler_char} has a positive answer, then Conjecture \ref{conj:zariski_right} also has a positive answer.

So, assume that Conjecture \ref{conj:Milnor_fiber_initial_form_Euler_char} has a positive answer. 

Let $f,g\colon(\mathbb{C}^n,0)\to (\mathbb{C},0)$ be two reduced holomorphic function-germs. Assume that there is a homeomorphism $\varphi\colon(\mathbb{C}^n,0)\to (\mathbb{C}^n,0)$ such that $f=g\circ \varphi$. We have to prove that $m(V(f),0)=m(V(g),0)$. 

Since Conjecture \ref{conj:Milnor_fiber_initial_form_Euler_char} has a positive answer, $\chi(F_{in(f)})=\chi(F_{in(g)})$. We have two cases to consider.

\noindent $\bullet$ {\bf Case $\chi(F_{in(f)})=\chi(F_{in(g)})\not =0$.} 

This implies $\chi(\C P^{n-1}\setminus PV(in(f)))\not=0$ and  $\chi(\C P^{n-1}\setminus PV(in(g)))\not=0$. By Theorem \ref{thm:acampo_complement},
$$
m(V(f),0)=\inf \{j>0;\Lambda(h_f^j)\not=0\}
$$
and 
$$
m(V(g),0)=\inf \{j>0;\Lambda(h_g^j)\not=0\}.
$$
By Theorem \ref{thm:monodromy}, $\Lambda(h_g^j)=\Lambda(h_f^j)$ for all positive integers $j$. Therefore, $m(V(f),0)=m(V(g),0)$.

\noindent $\bullet$ {\bf Case $\chi(F_{in(f)})=\chi(F_{in(g)})=0$.} 

For simplicity, let $m=m(V(f),0)$ and $k=m(V(g),0)$.
Assume by contradiction that $m\not= k$. We may assume that $m<k$. Since $m\geq 2$, we have that $k\geq 3$. 

Let us define $\tilde f, \tilde g\colon (\mathbb{C}^{n}\times\mathbb{C},0)\to (\mathbb{C},0)$ by
$\tilde f(x,t)=f(x)+t^k$ and $\tilde g(x,t)=g(x)+t^k$. Then, the homeomorphism $\tilde\varphi\colon(\mathbb{C}^n\times \C,0)\to (\mathbb{C}^n\times \C,0)$ given by $\tilde \varphi(x,t)=(\varphi(x),t)$ satisfies $\tilde f=\tilde g\circ \tilde \varphi$. 
Since we are assuming that Conjecture \ref{conj:Milnor_fiber_initial_form_Euler_char} has a positive answer, we have $\chi(F_{in(\tilde f)})=\chi(F_{in(\tilde g)})$. However, $in(\tilde f)(x,t)=in(f)(x)$ and  $in(\tilde g)(x,t)=in(g)(x)+t^k$. Then, $F_{in(\tilde f)}=F_{in(f)}\times \C$, and $F_{in(\tilde g)}$ is homotopy equivalent to the join of $F_{in(g)}$ and $F_k:=\{t\in \C; t^k=1\}$, denoted here by $F_{in(g)}\ast F_k$. Then, $\chi(F_{in(\tilde f)})=\chi(F_{in(f)})=0$. 

To calculate $\chi(F_{in(\tilde g)})$ we need the following result, which may be well-known, but we were unable to find a proper reference for it. Thus, we decided to present a proof of it.
\begin{lemma}
For any finite CW-complexes $X$ and $Y$, the Euler characteristic of their join satisfies:
\[
\chi(X * Y) =1-(\chi(X)-1)(\chi(Y)-1).
\]
\end{lemma}

\begin{proof}

It is a standard result that the join $X * Y$ is homotopy equivalent to the suspension of the smash product:
\begin{equation}\label{eq:relation_join-smash-suspension}
    X * Y \simeq \Sigma(X \wedge Y),
\end{equation}
where $X \wedge Y$ denotes the smash product and $\Sigma(Z)$ denotes the unreduced suspension of the topological space $Z$.
Recall that by definition, the smash product is the quotient space $X \wedge Y = (X \times Y) / (X \vee Y)$, where $X \vee Y$ denotes the wedge of $X$ and $Y$.

For finite CW-complexes, the Euler characteristic of a quotient pair $(A, B)$ satisfies $\chi(A/B) = \chi(A) - \chi(B) + 1$. Applying this property, we obtain:
\[
\chi(X \wedge Y) = \chi(X \times Y) - \chi(X \vee Y) + 1.
\]
Using that $\chi(X \times Y) = \chi(X)\chi(Y)$ and the additive property for wedge sums $\chi(X \vee Y) = \chi(X) + \chi(Y) - 1$, we substitute to find:
\begin{align*}
\chi(X \wedge Y) &= \chi(X)\chi(Y) - \big(\chi(X) + \chi(Y) - 1\big) + 1 \\
&= \chi(X)\chi(Y) - \chi(X) - \chi(Y) + 2.
\end{align*}
Therefore,
\begin{equation}\label{eq:smash_prod_formula}
\chi(X \wedge Y) = 1+(\chi(X)-1)(\chi(Y)-1).    
\end{equation}

Taking the unreduced suspension of a space shifts its reduced homology groups up by exactly one degree, meaning $\tilde{H}_{k}(\Sigma (Z)) \cong \tilde{H}_{k-1}(Z)$ for all $k$. Because the Euler characteristic is the alternating sum of the ranks of these homology groups, shifting the degrees by one flips the sign of the entire sum. Thus, for any space $Z$:
\[
\chi(\Sigma (Z)) -1= 1-\chi(Z).
\]
Combining this sign reversal with the homotopy equivalence $X * Y \simeq \Sigma(X \wedge Y)$ and Equation (\ref{eq:smash_prod_formula}), we conclude:
\begin{align*}
\chi(X * Y)-1 & = \chi\big(\Sigma(X \wedge Y)\big) -1\\
              & = 1 -\chi(X \wedge Y) \\
              & = 1-(1+(\chi(X)-1)(\chi(Y)-1))
\end{align*}
Therefore
\[
\chi(X * Y) =1-(\chi(X)-1)(\chi(Y)-1).
\]
\end{proof}
Then, 
$$
0=\chi(F_{in(\tilde f)})=\chi(F_{in(\tilde g)})=1-(\chi(F_{in(g)})-1)(\chi(F_k)-1)=\chi(F_k)=k\not =0,
$$
which is a contradiction. 

Therefore, we must have $m=k$.
\end{proof}

\subsection{Embedded topological invariance of the multiplicity of generalized semi-Brieskorn-Pham's hypersurfaces}
\begin{proposition}\label{prop:semi_top_mult}
Let $f,g\colon (\C^n,0)\to (\C,0)$ be two reduced holomorphic functions. Assume that $f$ and $g$ are generalized semi-Brieskorn-Pham functions and that there is a homeomorphism $\phi\colon (\C^n,0)\to (\C^n,0)$ such that $\phi(V(f))=V(g)$. If $m(V(f),0),m(V(g),0)\not=2$, then $m(V(f),0)=m(V(g),0)$ and $m_1(f)=m_1(g)$.
\end{proposition}
\begin{proof}
After a change of coordinates, if necessary, we may assume that $f$ and $g$ are written in the following way:
$$
f(z_1,...,z_r,z_{r+1},...,z_n)=z_1^{a}+...+z_r^{a}+h(z_{1},...,z_n),
$$
and 
$$
g(z_1,...,z_s,z_{s+1},...,z_n)=z_1^{b}+...+z_s^{b}+\tilde h(z_{1},...,z_n),
$$
where $h,\tilde h\colon (\C^{n},0)\to (\C,0)$ are holomorphic functions such that ${\rm ord}_0 h>a$ and ${\rm ord}_0 \tilde h>b$. In this case, $in(f)(z)= z_1^{a}+...+z_r^{a}$ and $in(g)(z)= z_1^{b}+...+z_s^{b}$, where $r=m_1(f)$ and $r=m_1(g)$. 

Note that if $r>1$, then $\chi(F_{in(f)})=1+(-1)^{r-1}(a-1)^r$, and if $r=1$, then $\chi(F_{in(f)})=a>0$. By the same reasoning, if $s>1$, then $\chi(F_{in(g)})=1+(-1)^{s-1}(b-1)^s$, and if $s=1$, then $\chi(F_{in(g)})=b>0$.

From Proposition \ref{prop:formula_euler_charac_monodromy}, $\Lambda(h_f^a)=\chi(F_{in(f)})$ and $\Lambda(h_g^b)=\chi(F_{in(g)})$. 
Since $a,b>2$, we have that $\Lambda(h_f^a)\not=0$ and $\Lambda(h_g^b)\not =0$. Therefore, by Theorem \ref{thm:acampo_complement}, $a=b$, and $\chi(F_{in(f)})=\chi(F_{in(g)})$. Thus, if $r,s>1$, then 
$$
1+(-1)^{r-1}(a-1)^r=1+(-1)^{s-1}(a-1)^s,
$$
which forces $r=s$.

Since $a>2$, $m=1$ is the unique integer solution for the following equation
$$
a=1+(-1)^{m-1}(a-1)^m.
$$
Thus, $r=1$ if and only if $s=1$.
\end{proof}

Since the multiplicity equal to 2 is an embedded topological invariant in $\C^3$ (see \cite{Aznar:1980} or \cite{FernandesJS:2025}), we obtain the following:

\begin{corollary}\label{cor:inv_mult_gsBP_surfaces}
Let $f,g\colon (\C^3,0)\to (\C,0)$ be two reduced holomorphic functions. Assume that $f$ and $g$ are generalized semi-Brieskorn-Pham functions and that there is a homeomorphism $\phi\colon (\C^3,0)\to (\C^3,0)$ such that $\phi(V(f))=V(g)$. Then $m(V(f),0)=m(V(g),0)$.
\end{corollary}

As another consequence, we obtain the following result, which was proved by Mendris and N\'emethi \cite[Corollary 2]{MendrisN:2005}.
\begin{corollary}\label{cor:inv_mult_Mendris-Nemethi_surfaces}
Let $f,g\colon (\C^2,0)\to (\C,0)$ be two irreducible holomorphic functions. Let $\tilde f,\tilde g\colon (\C^3,0)\to (\C,0)$ be the functions given by $\tilde f(x,y,z)=x^m+f(y,z)$ and $\tilde g(x,y,z)=x^k+g(y,z)$. Assume that there is a homeomorphism $\phi\colon (\C^3,0)\to (\C^3,0)$ such that $\phi(V(\tilde f))=V(\tilde g)$. Then $m(V(\tilde f),0)=m(V(\tilde g),0)$.
\end{corollary}

\begin{proposition}
Let $f,g\colon (\C^n,0)\to (\C,0)$ be two reduced holomorphic functions such that $C(V(f),0)$ is a linear subspace. Assume that there is a homeomorphism $\phi\colon (\C^n,0)\to (\C^n,0)$ such that $\phi(V(f))=V(g)$. Then $m(V(f),0)\geq m(V(g),0)$.
\end{proposition}
\begin{proof}
Since $C(V(f),0)$ is a linear subspace, $m_1(f)=1$. Thus, after a change of coordinates, if necessary, we may assume that $f$ is written in the following way:
$$
f(z_1,...,z_n)=z_1^{a}+h(z_{1},...,z_n),
$$
where $h\colon (\C^{n},0)\to (\C,0)$ is a holomorphic function such that ${\rm ord}_0 h>a$. In this case, $in(f)(z)= z_1^{a}$ and $\chi(F_{in(f)})=a>0$.

From Proposition \ref{prop:formula_euler_charac_monodromy}, $\Lambda(h_f^a)=\chi(F_{in(f)})$. Thus, $\Lambda(h_g^a)=\Lambda(h_f^a)=a\not=0$. Therefore, by Theorem \ref{thm:acampo_complement}, $a=m(V(f),0)\geq m(V(g),0)$.
\end{proof}

\begin{corollary}
Let $f,g\colon (\C^n,0)\to (\C,0)$ be two reduced holomorphic functions such that $C(V(f),0)$ and $C(V(g),0)$ are linear subspaces. Assume that there is a homeomorphism $\phi\colon (\C^n,0)\to (\C^n,0)$ such that $\phi(V(f))=V(g)$. Then $m(V(f),0)= m(V(g),0)$.
\end{corollary}

\begin{corollary}
Let $f,g\colon (\C^n,0)\to (\C,0)$ be two reduced holomorphic functions such that $C(V(f),0)$ and $C(V(g),0)$ are linear subspaces. Let $a_1,...,a_r,b_1,...,b_r\geq 3$ be integers, and $\tilde f,\tilde g\colon (\C^{n+r},0)\to (\C,0)$ be the functions $\tilde f(x_1,...,x_{n+1})=x_1^{a_1}+...+x_r^{a_r}+f(x_{r+1},...,x_{n+r})$ and $\tilde g(x_1,...,x_{n+r})=x_1^{b_1}+...+x_r^{b_r}+g(x_{r+1},...,x_{n+1})$. Assume that there is a homeomorphism $\phi\colon (\C^{n+1},0)\to (\C^{n+1},0)$ such that $\phi(V(\tilde f))=V(\tilde g)$. Then $m(V(\tilde f),0)=m(V(\tilde g),0)$.
\end{corollary}

\section{On the invariance of the corank}

\subsection{Topological invariance of the corank mod 2}
\begin{theorem}\label{thm:corank_mod_two}
Let $f,g\colon (\C^n,0)\to (\C,0)$ be germs of reduced holomorphic functions with a critical point at the origin. Assume that there is a homeomorphism $\varphi\colon(\C^n,0)\to (\C^n,0)$ such that $\varphi(V(f))=V(g)$. Then ${\rm crk}(f)\equiv {\rm crk}(g) \mod 2$. 
\end{theorem}
\begin{proof}
Instead of proving ${\rm crk}(f)\equiv {\rm crk}(g) \mod 2$, we are going to prove ${\rm rk}(f)\equiv {\rm rk}(g) \mod 2$.

Firstly, assume that ${\rm rk}(f)\equiv 1 \mod 2$. Then, after a holomorphic change of coordinates, we may assume that
$$
f(z_1,...,z_n)=z_1^2+...+z_{2k+1}^2+h(z_{1},...,z_n),
$$
where $h$ is holomorphic with ${\rm ord}_0h>2$. By Proposition \ref{prop:formula_euler_charac_monodromy},
\[ \Lambda(h_f^2) = \chi(F_{in(f)})=1+(-1)^{2k}=2\not=0. \]
By Theorem \ref{thm:acampo_complement}, $\Lambda(h_g^2)\not= 0$ and $m(V(g),0)\leq 2$. By A'Campo-L\^e's Theorem, $m(V(g),0)=2$. Then, after a holomorphic change of coordinates, we may assume that
$$
g(z_1,...,z_n)=z_1^2+...+z_{s}^2+\tilde h(z_{1},...,z_n),
$$
where $s={\rm rk}(f)$, and $\tilde h$ is holomorphic with ${\rm ord}_0\tilde h>2$. By Proposition \ref{prop:formula_euler_charac_monodromy},
\[ \Lambda(h_g^2) = \chi(F_{in(g)})=1+(-1)^{s-1}. \]
Since $\Lambda(h_g^2) =\Lambda(h_f^2) =2$, then $s$ is an odd number and thus ${\rm rk}(f)\equiv {\rm rk}(g) \mod 2$.

Now, assume that ${\rm rk}(f)\equiv 0 \mod 2$. Then, the first part of this proof forces ${\rm rk}(g)\equiv 0 \mod 2$. Then, ${\rm rk}(f)\equiv {\rm rk}(g) \mod 2$.

Therefore, in any case, ${\rm rk}(f)\equiv {\rm rk}(g) \mod 2$, and thus ${\rm crk}(f)\equiv {\rm crk}(g) \mod 2$.
\end{proof}

\subsection{Answer for the Arnold corank problem in dimensions 2 and 3}
The next result is a consequence of the topological invariance of the multiplicity in $\C^2$ and the Milnor number. However, here we present it as a consequence of Theorem \ref{thm:corank_mod_two}.
\begin{corollary}\label{cor:corank_dim_two}
Let $f,g\colon (\C^2,0)\to (\C,0)$ be germs of reduced holomorphic functions with a critical point at the origin. Assume that there is a homeomorphism $\varphi\colon(\C^2,0)\to (\C^2,0)$ such that $\varphi(V(f))=V(g)$. Then ${\rm crk}(f)= {\rm crk}(g)$. 
\end{corollary}
\begin{proof}
Let $r={\rm rk}(f)$ and $s={\rm rk}(g)$.

Since the multiplicity is an embedded topological invariant in $\C^2$, we have that ${\rm rk}(f)\not= 0$ if and only if ${\rm rk}(g)\not =0$. Thus, we may assume that $r,s\geq 1$, which implies $r,s\in \{1,2\}$.

By Theorem \ref{thm:corank_mod_two}, ${\rm crk}(f)\equiv {\rm crk}(g) \mod 2$. Then $r\equiv s \mod 2$. Since $r,s\in \{1,2\}$, we have $r=s$.
\end{proof}

Recently, Fernandes, Jelonek, and the author of this article in \cite{FernandesJS:2025} proved that the Arnold corank problem in dimension 3 has a positive answer. In this Section, we provide another proof of this result.

\begin{theorem}\label{thm:corank_dim_three}
Let $f,g\colon (\C^3,0)\to (\C,0)$ be germs of reduced holomorphic functions with a critical point at the origin. Assume that there is a homeomorphism $\varphi\colon(\C^3,0)\to (\C^3,0)$ such that $\varphi(V(f))=V(g)$. Then ${\rm crk}(f)= {\rm crk}(g)$. 
\end{theorem}
\begin{proof}
Since ${\rm crk}(f)=3-{\rm rk}(f)$ and ${\rm crk}(g) =3-{\rm rk}(g)$, it is enough to prove that ${\rm rk}(f)={\rm rk}(f)$. Let $r={\rm rk}(f)$ and $s={\rm rk}(g)$.

Since the multiplicity equal to 2 is an embedded topological invariant in $\C^3$ (see \cite{Aznar:1980} or \cite{FernandesJS:2025}), we have that ${\rm rk}(f)\not= 0$ if and only if ${\rm rk}(g)\not =0$. Thus, we may assume that $r,s\geq 1$. 

Suppose that $r=3$. Then, after a change of coordinates, if necessary, we may assume that $in(f)(z_1,z_2,z_3)= z_1^2+z_2^2+z_3^2$. Then $f$ has an isolated singularity at $0$, and $\mu(f)=1$. Then, $g$ also has an isolated singularity at $0$ and $\mu(g)=1$. Therefore, $s=3$ as well.

Thus, we may assume that $r,s\in \{1,2\}$.
By Theorem \ref{thm:corank_mod_two}, ${\rm crk}(f)\equiv {\rm crk}(g) \mod 2$. Then $r\equiv s \mod 2$. Therefore, if $r=1$ then $s=1$, and if $r=2$ then $s=2$.

Therefore, in any case, $r=s$.
\end{proof}

\subsection{Tangent Milnor fiber conjecture implies the Arnold corank problem for right equivalence}
\begin{theorem}
If Conjecture \ref{conj:Milnor_fiber_initial_form_homology} has a positive answer, then the Arnold corank problem for right topological equivalence also has a positive answer.
\end{theorem}
\begin{proof}
Let $f,g\colon (\C^n,0)\to (\C,0)$ be germs of reduced holomorphic functions with a critical point at the origin, and $\varphi\colon(\C^n,0)\to (\C^n,0)$ be a homeomorphism such that $f=g\circ \varphi$.

Assume that Conjecture \ref{conj:Milnor_fiber_initial_form_homology} has a positive answer. By \cite{Sampaio:2026}, $m(V(f),0)=m(V(g),0)$.
We may assume that $m(V(f),0)=m(V(g),0)=2$.
Let $r={\rm rk}(f)$ and $s={\rm rk}(g)$. Thus, $r,s\geq 1$.
Then $\widetilde{H}_{r-1}(in(f)^{-1}(1))=\mathbb{Z}$, $\widetilde{H}_{j}(in(f)^{-1}(1))=0$ for $j\not\in\{0,r-1\}$, and
$$
\widetilde{H}_0(in(f)^{-1}(1))=\left\{\begin{array}{ll}
\mathbb{Z},     & \mbox{if } r=1 \\
0,     &  \mbox{if } r>1.
\end{array}\right.
$$
Similarly, $\widetilde{H}_{s-1}(in(g)^{-1}(1))=\mathbb{Z}$, $\widetilde{H}_{j}(in(g)^{-1}(1))=0$ for $j\not\in\{0,s-1\}$, and 
$$
\widetilde{H}_0(in(g)^{-1}(1))=\left\{\begin{array}{ll}
\mathbb{Z},     & \mbox{if } s=1 \\
0,     &  \mbox{if } s>1.
\end{array}\right.
$$
Since Conjecture \ref{conj:Milnor_fiber_initial_form_homology} has a positive answer, we have 
$$
\widetilde{H}_j(in(f)^{-1}(1))\cong \widetilde{H}_j(in(g)^{-1}(1))\text{ for all }j. 
$$
Then $r=s$ and therefore ${\rm crk}(f)={\rm crk}(g).$
\end{proof}

\subsection{The Arnold corank problem has a positive answer for families with isolated singularities}
\begin{theorem}\label{thm:corank_family}
Let $\{f_t\}_t$ be a continuous family of holomorphic functions. If the Milnor number $\mu(f_t)$ is independent of $t$, finite, and positive, then the corank ${\rm crk}(f_t)$ is also independent of $t$.
\end{theorem}
\begin{proof} 
For each $t$, let $F_{t}=\{x\in \C^n;in(f_t)(x)=1\}$.
Let $r_t={\rm rk}(f_t)$. By \cite[Thereom 1.1]{BobadillaP:2022}, we may assume $r_t,r_0\geq 1$.

Then $\widetilde{H}^{r_t-1}(F_t)=\mathbb{Z}$, $\widetilde{H}^{j}(F_t)=0$ for $j\not\in\{0,r_t-1\}$, and
$$
\widetilde{H}^0(F_t)=\left\{\begin{array}{ll}
\mathbb{Z},     & \mbox{if } r_t=1 \\
0,     &  \mbox{if } r_t>1.
\end{array}\right.
$$
By \cite[Remark 8.2]{BobadillaP:2022}, $H_*^{BM}(F_t)=H_*^{BM}(F_0)$, where $H_*^{BM}(F)$ denotes the Borel-Moore homology of $F$. Since each $F_t$ is an oriented smooth manifold of real dimension $2n-2$, Poincaré duality gives an isomorphism modulo torsion between Borel-Moore homology and singular cohomology:
$$
H_*^{BM}(F_t)=H^{2n-2-*}(F_t).
$$
So, $H^{*}(F_t)=H^{*}(F_0)$, and then $r_t=r_0$. Therefore, ${\rm crk}(f_t)={\rm crk}(f_0).$
\end{proof}

We recall that the Arnold corank problem does not require the hypersurface singularities to be isolated. The two subsequent corollaries summarize some direct consequences of our results for non-isolated singularities.

Here, we are following closely \cite{BobadillaP:2022}.
        
In the setting of non-isolated singularities, a proper analogue of the Milnor number is the sequence of \emph{L\^e numbers}, introduced in \cite[Definition 2.1]{Massey:1990}. Thus, we have the following generalization of Theorem \ref{thm:corank_family}:
    
\begin{corollary}\label{cor:corank_non-isolated_one}
	Let  $f_t\colon (\C^n,0)\to (\C,0)$ be a continuous family of germs of holomorphic functions with a (possibly non-isolated) critical point at the origin. Assume that for some coordinate system on $\C^n$, all L\^e numbers of $f_t$ are independent of $t$.
	Then the corank of $f_t$ is independent of $t$.
\end{corollary}
\begin{proof} 
Applying inductively the uniform L\^e--Iomdine formulas \cite[Theorem 4.5]{Massey:1995}, we have that for some $j>2$, $k\in \{0,\dots, n\}$, and $a_1,\dots,a_k\in \C\setminus \{0\}$, the singularity $g_t=f_{t}+a_1z_1^{j}+\dots+a_kz_{k}^{j}$ is isolated, and its sequence of L\^e numbers, which in this case consists solely of the Milnor number $\mu(g_t)$ (see \cite[Example 2.1]{Massey:1995}), is independent of $t$.
    
 By Theorem \ref{thm:corank_family}, the corank of $g_t$ is independent of $t$. Since $j>2$, we have that ${\rm crk}(f_t)={\rm crk}(g_t)$; thus, the corank of $f_t$ is independent of $t$.
\end{proof}

Since L\^e numbers are not a topological invariant \cite[Theorem A(i)]{Bobadilla:2005}, Corollary \ref{cor:corank_non-isolated_one} does not imply a positive answer to the Arnold corank problem for families in general. 

However, it is known that for aligned singularities, the  L\^e numbers are topological invariants (see \cite[Corollary 7.8]{Massey:1995}). 
Recall that a germ of a hypersurface singularity $f\colon (\C^n,0)\to (\C,0)$ is {\bf aligned}, see \cite[Definition 7.1]{Massey:1995}, if its singular locus admits a stratification satisfying the Thom $a_{f}$ condition, such that the closure of each stratum is smooth at the origin. 
	
So, we obtain the following result for families of aligned singularities.
\begin{corollary}\label{cor:corank_non-isolated_two}
	Let  $f_t\colon (\C^n,0)\to (\C,0)$ be a continuous family of germs of holomorphic functions with a (possibly non-isolated) critical point at the origin. Assume that the embedded topological type of $f_t$ is independent of $t$, and each $f_t$ has aligned singularities.
	Then the corank of $f_t$ is independent of $t$.
\end{corollary}
\begin{proof}
By \cite[Corollary 7.8]{Massey:1995}, for any coordinate system on $\C^n$, all L\^e numbers of $f_t$ are independent of $t$. By Corollary \ref{cor:corank_non-isolated_one}, the corank of $f_t$ is independent of $t$.
\end{proof}

Theorem \ref{thm:corank_family} also implies an analogous statement in a purely algebraic setting.

\begin{corollary}\label{rem:algebraic_corank}
Let $k$ be a field of characteristic zero and $f\in k[t][\![ z_1,\dots,z_n ]\! ]$. Write $f_{t}=f(t,\cdot)\in k [\![ z_1,\dots,z_n ]\! ]$; so $t\mapsto f_{t}$ is an algebraic family of formal power series. Assume that the Milnor number of $f_{t}$ is finite, positive, and independent of $t$. Then the corank of $f_t$ is independent of $t$. 
\end{corollary}
\begin{proof}
For each $t$, the formal power series $f_t$ has a finite Milnor number $\mu$, implying it is finitely determined. Because the Milnor number is assumed to be independent of $t$, the degree of determinacy is uniformly bounded across the entire family. So, we can replace the formal power series $f$ by its polynomial truncation $f^{(N)} = \sum_{|I|<N} a_{I} z^{I}$ for some integer $N$ (for instance, $N = \mu + 2$) that is independent of $t$. Here, the coefficients $a_{I}$ belong to $k[t]$. This truncation is right-equivalent to the original series $f_t$; consequently, it preserves the Milnor number and the corank (which is completely determined by the $2$-jet of $f_t$).

The truncated family is defined by a finite collection of polynomials $a_I(t) \in k[t]$. Let $S \subset k$ be the finite set consisting of all coefficients of these polynomials $a_I$. We define $k_{0} = \Q(S)$ as the field extension of $\Q$ generated by the elements of $S$. Since $k_{0}$ is a finitely generated field extension of $\Q$, we have an embedding $k_{0}\hookrightarrow \C$. 

Now, the algebraic family $\{f_{t}\}_t$ can be treated as a family in $\C[t][z_1,\dots,z_n]$. Under this identification, $\{f_{t}\}_t$ constitutes a $\mu$-constant family of isolated complex hypersurface singularities. By Theorem \ref{thm:corank_family}, the corank of $f_{t}$ does not depend on $t$.

\end{proof}

\end{document}